\documentclass[12pt]{article}
\input{style}
\usepackage{enumitem}

\begin{document}

\title{Density of the set of probability measures with the martingale
  representation property}

\author{Dmitry Kramkov\footnote{Carnegie Mellon University, Department
    of Mathematical Sciences, 5000 Forbes Avenue, Pittsburgh, PA,
    15213-3890, US. The author also holds a part-time position at the
    University of Oxford. {\it Email: }kramkov@cmu.edu } \and Sergio
  Pulido\footnote{Laboratoire de Math\'ematiques et Mod\'elisation
    d'\'Evry (LaMME), Universit\'e d'\'Evry-Val-d'Essonne, ENSIIE,
    Universit\'e Paris-Saclay, UMR CNRS 8071, IBGBI 23 Boulevard de
    France, 91037 \'Evry Cedex, France. {\it Email:
    }sergio.pulidonino@ensiie.fr. The author's research benefited from
    the support of the Chair Markets in Transition (F\'ed\'eration
    Bancaire Fran\c caise) and the project ANR 11-LABX-0019.}}

\date{\today}

\maketitle

\begin{abstract}
  Let $\psi$ be a multi-dimensional random variable. We show that the
  set of probability measures $\mathbb{Q}$ such that the
  $\mathbb{Q}$-martingale $S^{\mathbb{Q}}_t=\cE{\mathbb{Q}}{t}{\psi}$
  has the Martingale Representation Property (MRP) is either empty or
  dense in $\mathcal{L}_\infty$-norm. The proof is based on a related
  result involving analytic fields of terminal conditions
  $(\psi(x))_{x\in U}$ and probability measures
  $(\mathbb{Q}(x))_{x\in U}$ over an open set $U$.  Namely, we show
  that the set of points $x\in U$ such that
  $S_t(x) = \cE{\mathbb{Q}(x)}{t}{\psi(x)}$ does not have the MRP,
  either coincides with $U$ or has Lebesgue measure zero. Our study is
  motivated by the problem of endogenous completeness in financial
  economics.
\end{abstract}

\begin{description}
\item[Keywords:] martingale representation property, martingales,
  stochastic integrals, analytic fields, endogenous completeness,
  complete market, equilibrium.
\item[AMS Subject Classification (2010):] 60G44, 60H05, 91B51, 91G99.
  % \item[JEL Classification:]
\end{description}

\section{Introduction}
\label{sec:introduction}

Let $(\Omega, \mathcal{F}, (\mathcal{F}_t), \mathbb{P})$ be a filtered
probability space, $\mathbb{Q}$ be an equivalent probability measure,
and $S = (S^i_t)$ be a multi-dimensional martingale under
$\mathbb{Q}$. It is often important to know whether $S$ has the
Martingale Representation Property (MRP), that is, whether every local
martingale under $\mathbb{Q}$ is a stochastic integral with respect to
$S$. For instance, in mathematical finance such MRP corresponds to the
\emph{completeness} of the market with stock prices $S$. By Jacod's
theorem, $S$ has the MRP if and only if $\mathbb{Q}$ is its only
equivalent martingale measure.

In many applications, $S$ is defined in a \emph{forward form}, as a
solution of an SDE, and the verification of the MRP is quite
straightforward. Suppose, for example, that $S$ is a $d$-dimensional
It\^o process such that
\begin{displaymath}
  dS_t = \sigma_t (\alpha_tdt + dB_t), 
\end{displaymath}
where $B$ is a $d$-dimensional Brownian motion, $\alpha =(\alpha_t)$
is a $d$-dimensional market price of risk process and
$\sigma = (\sigma_t)$ is a $d\times d$-dimensional volatility
process. Let us assume that the local martingale
\begin{displaymath}
  Z_t = \exp\left(-\int_0^t \alpha_s
    dB_s - \frac12 \int_0^t \abs{\alpha_s}^2 ds\right), \quad t\geq 0,  
\end{displaymath} 
is uniformly integrable; this fact can usually be verified by
Novikov's or Kazamaki's conditions.  By Girsanov's theorem, $Z$ is the
density process of an equivalent martingale measure $\mathbb{Q}$ for
$S$.  If the filtration is generated by $B$, then $S$ has the MRP
(equivalently, $\mathbb{Q}$ is its only equivalent martingale measure)
if and only if the matrix-valued volatility process
$\sigma=(\sigma_t)$ has full rank $d\mathbb{d\mathbb{P}} \times dt$
almost surely.

We are interested in the situation where both $S$ and $Z$ are
described in a \emph{backward form} through their terminal values:
\begin{equation}
  \label{eq:1}  
  \begin{split}
    Z_\infty &= \frac{d\mathbb{Q}}{d\mathbb{P}} =
    \frac{\zeta}{\mathbb{E}[\zeta]}, \\
    S_t &= \mathbb{E}^{\mathbb{Q}}[\psi|\mathcal{F}_t], \quad t\geq 0,
  \end{split}
\end{equation}
where $\zeta>0$ and $\psi = (\psi^i)$ are \emph{given} random
variables. Such setup naturally arises in the problem of
\emph{endogenous completeness} of financial economics, where the
random variable $\psi$ represents the terminal values of the traded
securities and $\mathbb{Q}$ defines an equilibrium pricing
measure. The term ``endogenous'' indicates that the stock prices
$S=(S^i)$ are \emph{computed} by~\eqref{eq:1} as part of the
solution. The examples include the construction of Radner
equilibrium~\cite{AnderRaim:08, HugMalTrub:12, RiedHerz:13, Kram:15}
and the verification of the completeness property for a market with
options~\cite{DavisObjoj:08, Schwarz:17}.

The main focus of the existing literature has been on the case when
the random variables $\zeta$ and $\psi$ are defined in terms of a
Markov diffusion in a form consistent with Feynman-Kac formula. The
proofs have relied on PDE methods and, in particular, on the theory of
analytic semigroups~\cite{KramPred:14}. A key role has been played by
the assumption that time-dependencies are analytic.

In this paper we do not impose any conditions on the form of the
random variables $\zeta$ and $\psi$. Our main results are stated as
Theorems~\ref{th:1} and~\ref{th:2}. In Theorem~\ref{th:1} we show that
the set
\begin{displaymath}
  \mathcal{Q}(\psi) \set \descr{\mathbb{Q}\sim
    \mathbb{P}}{S_t^{\mathbb{Q}}\set 
    \cE{\mathbb{Q}}{t}{\psi} \text{ has the MRP}}
\end{displaymath}
is either empty or $\mathcal{L}_{\infty}$-dense in the set of all
equivalent probability measures. In Theorem~\ref{th:2} we consider
analytic fields of probability measures $(\mathbb{Q}(x))_{x\in U}$ and
terminal conditions $(\psi(x))_{x\in U}$ over an open set $U$. We
prove that the exception set
\begin{displaymath}
  I \set \descr{x\in U}{S_t(x)\set\cE{\mathbb{Q}(x)}{t}{\psi(x)}
    \text{ does not have the MRP}}  
\end{displaymath}
either coincides with $U$ or has Lebesgue measure zero.

We expect the results of this paper to be useful in problems of
financial economics involving the endogenous completeness
property. One such application, to the problem of optimal investment
under price impact, is discussed in Remark~\ref{rem:2}.

\section{Density of the set of probability measures with the MRP}
\label{sec:main-result}

We work on a filtered probability space
$(\Omega, \mathcal{F}, (\mathcal{F}_t)_{t\geq 0}, \mathbb{P})$
satisfying the usual conditions of completeness and right-continuity;
the initial $\sigma$-algebra $\mathcal{F}_0$ is trivial and
$\mathcal{F} =\mathcal{F}_\infty$.  We denote by
$\mathcal{L}_1 = \mathcal{L}_1(\mathbf{R}^d)$ and
$\mathcal{L}_\infty = \mathcal{L}_\infty(\mathbf{R}^d)$ the Banach
spaces of (equivalence classes of) $d$-dimensional random variables
$\xi$ with the norms $\sNorm{\xi}{\mathcal{L}_1} \set \EP{\abs{\xi}}$
and
$\sNorm{\xi}{\mathcal{L}_\infty} \set
\inf\descr{c>0}{\ProbP{\abs{\xi}\leq c} = 1}$.
We use the same notation $\mathcal{L}_1$ for the isometric Banach space of
uniformly integrable martingales $M$ with the norm
$\sNorm{M}{\mathcal{L}_1} \set \sNorm{M_\infty}{\mathcal{L}_1}$.

For a matrix $A = (A^{ij})$ we denote its transpose by $A^*$ and
define its norm as
\begin{displaymath}
  \abs{A} \set  \sqrt{\trace{AA^*}} = \sqrt{\sum_{i,j}\abs{A^{ij}}^2}.
\end{displaymath}
If $X$ is a $m$-dimensional semimartingale and $\gamma$ is a
$m\times n$-dimensional $X$-integrable predictable process, then
$\gamma\cdot X = \int \gamma^* dX$ denotes the $n$-dimensional
stochastic integral of $\gamma$ with respect to $X$.  We recall that a
$n\times k$-dimensional predictable process $\zeta$ is
$(\gamma\cdot X)$-integrable if and only if $\gamma\zeta$ is
$X$-integrable.  In this case,
$\zeta \cdot (\gamma \cdot X) = (\gamma\zeta)\cdot X$ is a
$k$-dimensional semimartingale.

\begin{Definition}
  \label{def:1}
  Let $\mathbb{Q}$ be an equivalent probability measure
  ($\mathbb{Q}\sim \mathbb{P}$) and $S$ be a $d$-dimensional local
  martingale under $\mathbb{Q}$. We say that \emph{$S$ has the
    Martingale Representation Property (MRP)} if every local
  martingale $M$ under $\mathbb{Q}$ is a stochastic integral with
  respect to $S$, that is, there is a predictable $S$-integrable
  process $\gamma$ with values in $\mathbf{R}^d$ such that
  \begin{displaymath}
    M=M_0 + \gamma\cdot S. 
  \end{displaymath}
\end{Definition}

\begin{Remark}
  \label{rem:1}
  Jacod's theorem in~\cite[Section XI.1(a)]{Jacod:79} states that $S$
  has the MRP if and only if there is only one
  $\mathbb{Q}\sim \mathbb{P}$ such that $S$ is a local martingale
  under $\mathbb{Q}$.  Thus, there is no need to mention $\mathbb{Q}$
  in the definition of the MRP.
\end{Remark}

Let $\psi = (\psi^i)_{i=1,\dots,d}$ be a $d$-dimensional random
variable. We denote by $\mathcal{Q}(\psi)$ the family of probability
measures $\mathbb{Q}\sim \mathbb{P}$ such that
$\E{\mathbb{Q}}{\abs{\psi}} < \infty$ and the $\mathbb{Q}$-martingale
\begin{displaymath}
  S^{\mathbb{Q}}_t = \cE{\mathbb{Q}}{t}{\psi}, \quad t\geq 0, 
\end{displaymath}
has the MRP.

This is our first main result.

\begin{Theorem}
  \label{th:1}
  Suppose that $\psi\in\mathcal{L}_1(\mathbf{R}^d)$ and
  $\mathcal{Q}(\psi) \not=\emptyset$. Then for every $\epsilon>0$
  there is $\mathbb{Q}\in \mathcal{Q}(\psi)$ such that
  \begin{displaymath}
    \sNorm{\frac{d\mathbb{Q}}{d\mathbb{P}}-1}{\mathcal{L}_\infty} \leq
    \epsilon. 
  \end{displaymath}
\end{Theorem}

The proof is based on Theorem~\ref{th:2} from
Section~\ref{sec:analyticity-mr} and on the following elementary
lemma. We recall the definition of an analytic function with values in
a Banach space at the beginning of Section~\ref{sec:analyticity-mr}.

\begin{Lemma}
  \label{lem:1}
  Let $\zeta$ be a nonnegative random variable. Then the map
  $x\mapsto e^{-x\zeta}$ from $(0,\infty)$ to $\mathcal{L}_\infty$ is
  analytic.
\end{Lemma}
\begin{proof}
  Fix $y>0$. For every $\omega\in \Omega$ the function
  $x\mapsto e^{-x\zeta(\omega)}$ has a Taylor's expansion
  \begin{equation}
    \label{eq:2}
    e^{-x\zeta(\omega)} = \sum_{n=0}^\infty A_n(y)(\omega) (x-y)^n,
    \quad x\in \mathbf{R},
  \end{equation}
  where
  \begin{displaymath}
    A_n(y) = \frac1{n!} \frac{d^n}{dx^n} \left(e^{-x\zeta}\right)|_{x=y} =
    \frac1{n!} (-1)^n \zeta^n e^{-y\zeta}. 
  \end{displaymath}
  We deduce that
  \begin{displaymath}
    \norm{A_n(y)}_{\mathcal{L}_{\infty}} \leq \frac1{n!} \max_{t\geq 0}
    (t^n e^{-yt}) = \frac1{n!} \left(\frac{n}{ey}\right)^n  \leq K
    \frac1{\sqrt{n}} \left(\frac1{y}\right)^n,
  \end{displaymath}
  where the existence of a constant $K>0$ follows from Stirling's
  formula:
  \begin{displaymath}
    \lim_{n\to \infty} \frac{\sqrt{2\pi n}}{n!} \left(\frac{n}{e}\right)^n = 1. 
  \end{displaymath}
  It follows that the series in~\eqref{eq:2} converges in
  $\mathcal{L}_\infty$ provided that $\abs{x-y}<y$.
\end{proof}

\begin{proof}[Proof of Theorem~\ref{th:1}]
  We take $\mathbb{R}\in \mathcal{Q}(\psi)$, denote
  $\zeta \set \frac{d\mathbb{R}}{d\mathbb{P}}$, and for $x>0$ define
  the random variables
  \begin{align*}
    \zeta(x) &\set \frac{1-e^{-x\zeta}}{x} + \frac{x}{1+x}, \\
    \xi(x) &\set \zeta(x) \psi,
  \end{align*}
  and a probability measure $\mathbb{Q}(x)$ such that
  \begin{displaymath}
    \frac{d\mathbb{Q}(x)}{d\mathbb{P}} = \frac{\zeta(x)}{\EP{\zeta(x)}}. 
  \end{displaymath}

  We set $\zeta(0) \set \zeta$, $\xi(0) \set \zeta\psi$, and
  $\mathbb{Q}(0)\set \mathbb{R}$ and observe that for every
  $\omega \in \Omega$ the functions $x\mapsto \zeta(x)(\omega)$ and
  $x\mapsto \xi(x)(\omega)$ on $[0,\infty)$ are continuous.  Since
  \begin{displaymath}
    \abs{\zeta(x)} \leq \zeta \sup_{t\geq 0} \frac{1-e^{-t}}{t} +
    \frac{x}{1+x} \leq \zeta + 1, 
  \end{displaymath}
  the dominated convergence theorem yields that $x\mapsto \zeta(x)$
  and $x\mapsto \xi(x)$ are continuous maps from $[0,\infty)$ to
  $\mathcal{L}_1$.  By Lemma~\ref{lem:1}, $x\mapsto \zeta(x)$ is an
  analytic map from $(0, \infty)$ to $ \mathcal{L}_\infty$ and thus
  $x\mapsto \zeta(x)$ and $x\mapsto \xi(x)$ are analytic maps from
  $(0,\infty)$ to $\mathcal{L}_1$.  Theorem~\ref{th:2} then implies
  that the exception set
  \begin{displaymath}
    I \set \descr{x>0}{\mathbb{Q}(x)\not\in \mathcal{Q}(\psi)}
  \end{displaymath}
  is at most countable.

  Choose now any $\epsilon>0$. Since
  \begin{displaymath}
    -\frac1{1+x} \leq \zeta(x) -1 \leq \frac1x - \frac1{1+x},
  \end{displaymath}
  there is $x_0 = x_0(\epsilon)$ such that the assertion of the
  theorem holds for every $\mathbb{Q}(x)$ with $x\geq x_0$ and
  $x\not \in I$.
\end{proof}

\begin{Remark}
  \label{rem:2}
  Theorem~\ref{th:1} plays a key role in our work, in progress, on the
  problem of optimal investment in a ``backward'' model of price
  impact~\cite{Germ:11,KramkovPulido:16}.  There are a large investor
  with utility function $U=U(x)$ and initial capital $X_0$ and
  a market maker with exponential utility function
  \begin{displaymath}
    V(y) = \frac1a\left(1-e^{-ay}\right), \quad y\in \mathbf{R}, 
  \end{displaymath}
  where $a>0$ is the absolute risk-aversion coefficient.  The investor
  looks for a predictable process $\gamma = (\gamma_t)$ of the numbers
  of stocks that maximizes the expected utility:
  \begin{displaymath}
    u(X_0) = \sup_{\gamma} \EP{U(X_0 + \gamma \cdot S(\gamma)_T)}.
  \end{displaymath}
  While the terminal stock prices are fixed to random dividends
  $\psi$:
  \begin{displaymath}
    S_T(\gamma) = \psi, 
  \end{displaymath}
  their intermediate values are set so that the opposite position to
  the demand $\gamma$ is optimal for the market maker:
  \begin{displaymath}
    -\gamma = \argmax_{\zeta}  \EP{V(\zeta \cdot S(\gamma)_T)}.
  \end{displaymath}
  The standard first-order conditions in optimal investment lead to
  the expression for prices $S(\gamma)$ in a \emph{backward form}:
  \begin{displaymath}
    S_t(\gamma)=\cE{\mathbb{Q(\gamma)}}{t}{\psi},
  \end{displaymath}
  where
  \begin{displaymath}
    \frac{d\mathbb{Q(\gamma)}}{d\mathbb{P}}
    =\frac{V'(-\gamma \cdot S(\gamma)_T)}{\EP{V'(-\gamma\cdot
        S(\gamma)_T)}} = 
    \frac{\exp \left(a \gamma \cdot
        S(\gamma)_T\right)}{\EP{\exp\left(a \gamma \cdot 
          S(\gamma)_T\right)}}.    
  \end{displaymath}

  Theorem~\ref{th:1} allows us to relax this apparently complex
  stochastic control problem into a simple \emph{static} framework.
  More precisely, we show that
  \begin{displaymath}
    u(X_0) = \max_{\xi \in \mathcal{C}} \EP{U(X_0 + \xi)}, 
  \end{displaymath}
  where $\mathcal{C}$ is the family of random variables given by
  \begin{displaymath}
    \mathcal{C} \set \descr{\xi}{\EP{\xi V'(-\xi)}=0} =
    \descr{\xi}{\EP{\xi e^{a\xi}}=0}.
  \end{displaymath}
  The main ingredient of the proof is the assertion that the family of
  terminal gains of trading strategies
  \begin{displaymath}
    \mathcal{D} \set \descr{\xi}{\xi = \gamma \cdot S(\gamma)_T \text{
        for some demand $\gamma$}} 
  \end{displaymath}
  is $\mathcal{L}_\infty$-dense in $\mathcal{C}$, which can be
  interpreted as an {\it approximate completeness} of the model.  This
  claim follows from Theorem~\ref{th:1}, after we observe that a
  random variable $\xi\in \mathcal{C}$ also belongs to
  $\mathcal{D}$ if the $\mathbb{Q}(\xi)$-martingale $S(\xi)$ has the
  MRP, where
  \begin{align*}
    S_t(\xi) &= \cE{\mathbb{Q(\xi)}}{t}{\psi}, \\
    \frac{d\mathbb{Q(\xi)}}{d\mathbb{P}} 
             &= \frac{V'(-\xi)}{\EP{V'(-\xi)}} =
               \frac{\exp\left(a\xi\right)}{\EP{\exp\left({a\xi}\right)}}.     
  \end{align*}
\end{Remark}

\section{The MRP for analytic fields of martingales}
\label{sec:analyticity-mr}

Let $\mathbf{X}$ be a Banach space and $U$ be an open connected set in
$\mathbf{R}^d$. We recall that a map $x\mapsto X(x)$ from $U$ to
$\mathbf{X}$ is \emph{analytic} if for every $y\in U$ there exist a
number $\epsilon = \epsilon(y)>0$ and elements $(Y_\alpha(y))$ in
$\mathbf{X}$ such that the $\epsilon$-neighborhood of $y$ belongs to
$U$ and
\begin{displaymath}
  X(x) = \sum_\alpha Y_\alpha(y) (x-y)^\alpha, \quad \abs{y-x}<\epsilon. 
\end{displaymath}
Here the series converges in the norm $\norm{\cdot}_{\mathbf{X}}$ of
$\mathbf{X}$, the summation is taken with respect to multi-indices
$\alpha = (\alpha_1, \dots,\alpha_d) \in \mathbf{Z}_+^l$ of
non-negative integers, and if $x = (x_1,\dots,x_d) \in \mathbf{R}^d$,
then $x^\alpha \set \prod_{i=1}^d x_i^{\alpha_i}$.

This is our second main result.

\begin{Theorem}
  \label{th:2}
  Let $U$ be an open connected set in $\mathbf{R}^l$ and suppose that
  the point $x_0\in\mathbf{R}^l$ belongs to the closure of $U$. Let
  $x\mapsto \zeta(x)$ and $x\mapsto \xi(x)$ be continuous maps from
  ${U}\cup \braces{x_0}$ to $\mathcal{L}_1(\mathbf{R})$ and
  $\mathcal{L}_1(\mathbf{R}^d)$, respectively, whose restrictions to
  $U$ are analytic. For every $x\in {U}\cup \braces{x_0}$, assume that
  $\zeta(x)>0$ and define a probability measure $\mathbb{Q}(x)$ and a
  $\mathbb{Q}(x)$-martingale $S(x)$ by
  \begin{align*}
    \frac{d\mathbb{Q}(x)}{d\mathbb{P}} =
    \frac{\zeta(x)}{\EP{\zeta(x)}}, \quad S_t(x) =
    \cE{\mathbb{Q}(x)}{t}{\frac{\xi(x)}{\zeta(x)}}.
  \end{align*}
  If the $\mathbb{Q}(x_0)$-martingale $S(x_0)$ has the MRP, then the
  exception set
  \begin{displaymath}
    I \set \descr{x\in U}{\text{the $\mathbb{Q}(x)$-martingale }S(x)
      \text{ does not have the MRP}}  
  \end{displaymath}
  has Lebesgue measure zero. If, in addition, $U$ is an interval in
  $\mathbf{R}$, then the set $I$ is at most countable.
\end{Theorem}

The following example shows that \emph{any} countable set $I$ in
$\mathbf{R}$ can play the role of the exception set of
Theorem~\ref{th:2}. In this example we choose $\zeta(x) =1$ (so that
$\mathbb{Q}(x)=\mathbb{P}$) and take $x\mapsto \xi(x)$ to be a
\emph{linear} map from $\mathbf{R}$ to
$\mathcal{L}_\infty(\mathbf{R})$.

\begin{Example}
  \label{ex:1}
  Let $(\Omega,\mathcal{F},(\mathcal{F}_n), \mathbb{P})$ be a filtered
  probability space, where the filtration is generated by independent
  Bernoulli random variables $(\epsilon_n)$ with
  \begin{displaymath}
    \ProbP{\epsilon_n=1} = \ProbP{\epsilon_n=-1} = \frac12. 
  \end{displaymath}
  It is well known that every martingale $(N_n)$ admits the unique
  ``integral'' representation:
  \begin{equation}
    \label{eq:3}
    N_n = N_0 + \sum_{k=1}^n h_k(\epsilon_1, \dots,\epsilon_{k-1})
    \epsilon_k, 
  \end{equation}
  for some functions $h_k = h_k(x_1,\dots,x_{k-1})$, $k\geq 1$, where
  $h_1$ is just a constant.

  Let $I=(x_n)$ be an arbitrary sequence in $\mathbf{R}$. We define a
  linear map $x\mapsto \xi(x)$ from $\mathbf{R}$ to
  $\mathcal{L}_\infty(\mathbf{R})$ by
  \begin{displaymath}
    \xi(x)=\sum_{n=1}^{\infty}
    \frac{(x-x_n)}{2^n(1+\abs{x_n})} \epsilon_n = \psi_0 + \psi_1 x,  
  \end{displaymath}
  where $\psi_0$ and $\psi_1$ are bounded random variables:
  \begin{displaymath}
    \psi_0= - \sum_{n=1}^{\infty}\frac{x_n}{2^n(1+\abs{x_n})}
    \epsilon_n,\quad
    \psi_1=\sum_{n=1}^{\infty}
    \frac{1}{2^n(1+\abs{x_n})} \epsilon_n. 
  \end{displaymath} 
  We have that
  \begin{displaymath}
    M_n(x) = \cEP{n}{\xi(x)} = \mathbb{E}\left[\xi(x) \lvert
      \epsilon_1,\dots,\epsilon_n\right] = \sum_{k=1}^{n}
    \frac{(x-x_k)}{2^k(1+\abs{x_k})} \epsilon_k
  \end{displaymath}
  and thus
  \begin{displaymath}
    \Delta M_n(x) = M_n(x) - M_{n-1}(x) =
    \frac{(x-x_n)}{2^n(1+\abs{x_n})} \epsilon_n. 
  \end{displaymath}
  If $x\not\in I$, then the martingale $(N_n)$ from~\eqref{eq:3} is a
  stochastic integral with respect to $M(x)$:
  \begin{displaymath}
    N_n = N_0 + \sum_{k=1}^n
    h_k(\epsilon_1,\dots,\epsilon_{k-1})\frac{2^k(1+\abs{x_k})}{(x-x_k)}
    \Delta M_k(x).
  \end{displaymath}
  However, if $x_m \in I$, then the martingales $M(x_m)$ and
  \begin{displaymath}
    L^{(m)}_n = \sum_{k=1}^n \ind{k=m}\epsilon_k = \ind{n\geq
      m}\epsilon_m, \quad n\geq 0, 
  \end{displaymath}
  are orthogonal. Hence, $L^{(m)}$ does not admit an integral
  representation with respect to $M(x_m)$.
\end{Example}

The rest of the section is devoted to the proof of
Theorem~\ref{th:2}. It relies on Theorems~\ref{th:3} and \ref{th:4}
from the appendices and on the lemmas below.

Throughout the paper all operations on stochastic processes are
defined pointwise, for every $(t,\omega)$. In particular, if $X$ is a
matrix-valued process, then $\abs{X}$ denotes the one-dimensional
process of the running norm:
\begin{displaymath}
  \abs{X}_t(\omega) \set \abs{X_t(\omega)}. 
\end{displaymath}

Let $X$ be a (uniformly) square integrable martingale taking values in
$\mathbf{R}^m$. We denote by $\qvar{X} = ([X^i,X^j])$ its process of
quadratic variation and by $\pvar{X} = (\pcov{X^i}{X^j})$ its
predictable process of quadratic variation; they both take values in
the cone $\mathcal{S}^m_+$ of symmetric nonnegative
$m\times m$-matrices. We define the predictable increasing process
\begin{displaymath}
  A^X \set \trace{\pvar{X}} = \sum_{i=1}^m \pcov{X^i}{X^i}.
\end{displaymath}
Standard arguments show that there is a predictable process
$\kappa^{X}$ with values in $\mathcal{S}^m_+$ such that
\begin{displaymath}
  \pvar{X} = (\kappa^{X})^2 \cdot A^X.
\end{displaymath}

On the predictable $\sigma$-algebra $\mathcal{P}$ of
$[0,\infty) \times \Omega$ we introduce a measure
\begin{displaymath}
  \mu^{X}(dt,d\omega) \set dA^X_t(\omega) \ProbP{d\omega}.
\end{displaymath}
For a nonnegative predictable process $\gamma$ the expectation under
$\mu^{X}$ is given by
\begin{displaymath}
  \E{\mu^{X}}{\gamma} = \EP{\int_0^\infty \gamma dA^X} =
  \EP{\int_0^\infty \gamma_t dA^X_t}. 
\end{displaymath}
We observe that this measure is finite:
\begin{displaymath}
  \mu^{X}([0,\infty)\times \Omega) = \EP{A^X_\infty} =
  \EP{\abs{X_\infty - X_0}^2}<\infty.
\end{displaymath}
For predictable $m$-dimensional processes $(\gamma^n)$ and $\gamma$
the notation $\gamma^n \overset{\mu^{X}}{\rightarrow} \gamma$ stands
for the convergence in measure $\mu^{X}$:
\begin{displaymath}
  \forall \epsilon >0: \quad \mM{\mu^{X}}{\abs{\gamma^n -
      \gamma}>\epsilon} 
  \to 0, \quad n\to \infty.  
\end{displaymath} 

\begin{Lemma}
  \label{lem:2}
  Let $X$ be a square integrable martingale with values in
  $\mathbf{R}^m$ and $\gamma$ be a predictable $m$-dimensional
  process. Then $\gamma$ is $X$-integrable and $\gamma\cdot X =0$ if
  and only if $\kappa^{X}\gamma = 0$, $\mu^{X}-a.s.$.
\end{Lemma}
\begin{proof}
  Since $\gamma \ind{\abs{\gamma}\leq n}\cdot X \to \gamma\cdot X$ as
  $n\to \infty$ in the semimartingale topology, we can assume without
  a loss in generality that $\gamma$ is bounded. Then $\gamma\cdot X$
  is a square integrable martingale with predictable quadratic
  variation
  \begin{displaymath}
    \langle \gamma\cdot X \rangle_t =
    \int_0^t\abs{{\kappa^{X}\gamma}}^2dA^X =
    \int_0^t\abs{{\kappa^{X}_s\gamma}_s}^2dA^X_s 
  \end{displaymath}
  and the result follows from the identity:
  \begin{displaymath}
    \EP{(\gamma\cdot X)_\infty^2} = \EP{\langle \gamma\cdot X
      \rangle_\infty} = \EP{\int_0^\infty \abs{{\kappa^{X}\gamma}}^2 dA^X} =
    \E{\mu^{X}}{\abs{{\kappa^{X}\gamma}}^2}. 
  \end{displaymath}
\end{proof}

For every predictable process $\zeta$ taking values in
$\mathcal{S}^m_+$ we can naturally define a $\mathcal{S}^m_+$-valued
predictable process $\zeta^\oplus$ such that for all $(t,\omega)$ the
matrix $\zeta^\oplus_t(\omega)$ is the pseudo-inverse to the matrix
$\zeta_t(\omega)$.

From Lemma~\ref{lem:2} we deduce that if $\alpha$ is an integrand for
$X$ then the predictable process
\begin{displaymath}
  \beta \set {\kappa^X}^{\oplus}{\kappa^X} \alpha
\end{displaymath}
is also $X$-integrable and $\alpha\cdot X = \beta\cdot X$. Moreover,
$\abs{\beta} \leq \abs{\alpha}$, by the minimal norm property of the
pseudo-inverse matrices. In view of this property, we call a
predictable $m$-dimensional process $\gamma$ a \emph{minimal
  integrand} for $X$ if $\gamma$ is $X$-integrable and
\begin{displaymath}
  \gamma =  {\kappa^X}^{\oplus}{\kappa^X} \gamma. 
\end{displaymath}
From the definition of a minimal integrand we immediately deduce that
\begin{equation}
  \label{eq:4}
  \abs{{\kappa^X}\gamma} \leq \abs{{\kappa^X}}\abs{\gamma},\quad
  \abs{\gamma} \leq \abs{{\kappa^X}^{\oplus}} \abs{{\kappa^X}\gamma}, 
\end{equation}
where, following our convention, both the norm and the inequalities
are defined pointwise, for every $(t,\omega)$.

We denote by $\mathcal{H}_1=\mathcal{H}_1(\mathbf{R}^d)$ the Banach
space of uniformly integrable $d$-dimensional martingales $M$ with the
norm:
\begin{align*}
  \sNorm{M}{\mathcal{H}_1} &\set \EP{\sup_{t\geq 0} \abs{M_t}}.
\end{align*}
By Davis' inequality, the convergence $M^n\to 0$ in $\mathcal{H}_1$ is
equivalent to the convergence $\qvar{M^n}_\infty^{1/2} \to 0$ in
$\mathcal{L}_1$, where $\qvar{M^n}$ is the quadratic variation process
of $M^n$.

We say that a sequence $(N^n)$ of local martingales converges to a
local martingale $N$ in $\mathcal{H}_{1,loc}$ if there are stopping
times $(\tau^m)$ such that $\tau^m\uparrow \infty$ and
$N^{n,\tau^m}\to N^{\tau^m}$ in $\mathcal{H}_1$. Here as usual, we
write $Y^{\tau} \set (Y_{\min(t,\tau)})$ for a semimartingale $Y$
stopped at a stopping time $\tau$.

\begin{Lemma}
  \label{lem:3}
  Let $X$ be a square integrable martingale with values in
  $\mathbf{R}^m$ and $(\gamma^n)$ be a sequence of predictable
  $m$-dimensional $X$-integrable processes such that the stochastic
  integrals $(\gamma^n\cdot X)$ converge to $0$ in
  $\mathcal{H}_{1,loc}$. Then
  ${\kappa^{X}}\gamma^n \overset{\mu^{X}}{\rightarrow} 0$. If, in
  addition, $(\gamma^n)$ are minimal integrands then
  $\gamma^n \overset{\mu^{X}}{\rightarrow} 0$.
\end{Lemma}

\begin{proof}
  It is sufficient to consider the case of minimal integrands.  By
  localization, we can suppose that $\gamma^n\cdot X \to 0$ in
  $\mathcal{H}_1$, which by Davis' inequality is equivalent to the
  convergence of $(\qvar{\gamma^n\cdot X}_\infty^{1/2})$ to $0$ in
  $\mathcal{L}_1$.

  Assume for a moment that $\abs{\gamma^n}\leq 1$.  Then
  $\qvar{\gamma^n\cdot X} \leq \qvar{X}$ and the theorem of dominated
  convergence yields that $\qvar{\gamma^n\cdot X}_\infty \to 0$ in
  $\mathcal{L}_1$. As
  \begin{displaymath}
    \EP{\qvar{\gamma^n\cdot X}_\infty} =      \EP{\pvar{\gamma^n\cdot
        X}_\infty} =  \EP{\int_0^\infty \abs{\kappa^X \gamma^n}^2
      dA^X} = \E{\mu^{X}}{\abs{{\kappa^X}\gamma^n}^2},
  \end{displaymath}
  we deduce that
  ${\kappa^{X}}\gamma^n \overset{\mu^{X}}{\rightarrow} 0$, which in
  view of~\eqref{eq:4}, also implies that
  $\gamma^n \overset{\mu^{X}}{\rightarrow} 0$.

  In the general case, we observe that
  \begin{displaymath}
    \beta^n \set \frac1{1+\abs{\gamma^n}}\gamma^n
  \end{displaymath}
  are minimal integrands for $X$ such that $\abs{\beta^n}\leq 1$ and
  $\qvar{\beta^n\cdot X} \leq \qvar{\gamma^n\cdot X}$.  Hence, by what
  we have already proved, $\beta^n \overset{\mu^{X}}{\rightarrow} 0$,
  which clearly yields that
  $\gamma^n \overset{\mu^{X}}{\rightarrow} 0$ and then that
  ${\kappa^{X}}\gamma^n \overset{\mu^{X}}{\rightarrow} 0$.
\end{proof}

\begin{Lemma}
  \label{lem:4}
  Let $X$ be a square integrable $m$-dimensional martingale and
  $\gamma = (\gamma^{ij})$ be a predictable $X$-integrable process
  with values in $\mathbf{R}^{m\times d}$. Then $X$ is a stochastic
  integral with respect to $Y\set \gamma\cdot X$, that is
  $X = X_0 + \zeta\cdot Y$ for some predictable $Y$-integrable
  $d\times m$-dimensional process $\zeta$, if and only if
  \begin{equation}
    \label{eq:5}
    \rank{{\kappa^{X}}\gamma} = \rank{\kappa^{X}}, \quad
    \mu^{X}-a.s..  
  \end{equation}
\end{Lemma}

\begin{proof}
  We recall that a predictable process $\zeta$ is
  $Y = \gamma\cdot X$-integrable if and only if $\gamma \zeta$ is
  $X$-integrable. From Lemma~\ref{lem:2} we deduce that $\zeta$ is
  $Y$-integrable and satisfies
  \begin{displaymath}
    X = X_0 + \zeta\cdot Y = X_0 + \zeta\cdot (\gamma \cdot X) =
    (\gamma\zeta)\cdot X
  \end{displaymath}
  if and only if
  \begin{displaymath}
    {\kappa^{X}}\gamma \zeta = {\kappa^{X}}, \quad \mu^{X}-a.s.. 
  \end{displaymath}
  However, the solvability of this linear equation with respect to
  $\zeta$ is equivalent to~\eqref{eq:5} by an elementary argument from
  linear algebra.
\end{proof}

\begin{Lemma}
  \label{lem:5}
  Let $U$ be an open connected set in $\mathbf{R}^d$ and
  $x\mapsto \sigma(x)$ be an analytic map with values in
  $k\times l$-matrices. Then there is a nonzero real-analytic function
  $f$ on $U$ such that
  \begin{displaymath}
    E \set \descr{x\in U}{\rank \sigma(x) < \sup_{y\in U} \rank
      \sigma(y)}= \descr{x\in U}{f(x)=0}. 
  \end{displaymath}
  In particular, the set $E$ has Lebesgue measure zero and if $d=1$,
  then it consists of isolated points.
\end{Lemma}

\begin{proof}
  Let $m \set \sup_{y\in U} \rank \sigma(y)$. If $m=0$, then the set
  $E$ is empty and we can take $f=1$. If $m>0$, then the result holds
  for
  \begin{displaymath}
    f(x) = \sum_{\alpha} \det \sigma_\alpha(x)\sigma^*_\alpha(x), 
  \end{displaymath}
  where $(\sigma_\alpha)$ is the family of all $m\times m$
  sub-matrices of $\sigma$. The remaining assertions follow from the
  well-known properties of zero-sets of real-analytic functions.
\end{proof}

\begin{proof}[Proof of Theorem~\ref{th:2}]
  Without restricting generality we can assume that $\zeta(x_0)=1$
  and, hence, $\mathbb{Q}(x_0)=\mathbb{P}$.  Proposition~2
  in~\cite{KramSirb:06} shows that if some multi-dimensional local
  martingale has the MRP, then there is a {bounded}, hence square
  integrable, $m$-dimensional martingale $X$ that has the MRP.  We fix
  such $X$ and use for it the $\mathcal{S}^m_+$-valued predictable
  process $\kappa^{X}$ and the finite measure $\mu^{X}$ on the
  predictable $\sigma$-algebra $\mathcal{P}$ introduced just before
  Lemma~\ref{lem:2}.

  We define the martingales
  \begin{align*}
    Y_t(x) \set \cEP{t}{\zeta(x)}, \quad R_t(x) \set \cEP{t}{\xi(x)},
  \end{align*}
  and observe that $R(x) = S(x)Y(x)$.  Let $\alpha(x)$ and $\beta(x)$
  be integrands for $X$ with values in $\mathbf{R}^{m}$ and
  $\mathbf{R}^{m\times d}$, respectively, such that
  \begin{align*}
    Y(x) &= Y_0(x) + Y_{-}(x)\alpha(x) \cdot X, \\
    R(x) &= R_0(x) + Y_{-}(x)\beta(x) \cdot X, 
  \end{align*}
  where as usual, $Y_{-}$ stands for the left-continuous process
  $(Y_{t-})$.  Integration by parts yields that
  \begin{align*}
    dR(x) - S_{-}(x) dY(x) = Y_{-}(x)d(S(x) +
    \qcov{S(x)}{\alpha(x)\cdot X}).
  \end{align*}
  It follows that
  \begin{displaymath}
    S(x) + \qcov{S(x)}{\alpha(x)\cdot X} = S_0(x) + \sigma(x)\cdot X,  
  \end{displaymath}
  where
  \begin{align*}
    \sigma(x) = \beta(x) - \alpha(x) S^*_{-}(x).
  \end{align*}
  From Theorem~\ref{th:4} we deduce that $S(x)$ has the MRP (under
  $\mathbb{Q}(x)$) if and only if the stochastic integral
  $\sigma(x)\cdot X$ has the MRP. By Lemma~\ref{lem:4} the latter
  property is equivalent to
  \begin{displaymath}
    \rank {\kappa^{X}}\sigma(x) = \rank \kappa^{X}, \quad
    \mu^{X}-a.s.,    
  \end{displaymath}
  and therefore, the exception set $I$ admits the description:
  \begin{displaymath}
    I = \descr{x\in U}{\mM{\mu^{X}}{D(x)}>0},  
  \end{displaymath}
  where for $x\in U\cup \braces{x_0}$ the predictable set $D(x)$ is
  given by
  \begin{displaymath}
    D(x) = \descr{(t,\omega)}{\rank
      {\kappa^{X}_t(\omega)}\sigma_t(x)(\omega) < 
      \rank \kappa^{X}_t(\omega)}. 
  \end{displaymath}

  From Theorem~\ref{th:3} we deduce the existence of the integrands
  $\alpha(x)$ and $\beta(x)$ and of the modifications of the
  martingales $Y(x)$ and $R(x)$ such that for every
  $(t,\omega)\in [0,\infty)\times \Omega$ the function
  \begin{displaymath}
    x\mapsto \sigma_t(x)(\omega) = \beta_t(x)(\omega) -
    \alpha_t(x)(\omega) \frac{R^*_{t-}(x)(\omega)}{Y_{t-}(x)(\omega)},
  \end{displaymath}
  taking values in the space of $m\times d$-matrices, is analytic on
  $U$. Hereafter, we shall use these versions.

  Let $\lambda$ be the Lebesgue measure on $\mathbf{R}^l$ and
  $\mathcal{B} = \mathcal{B}(U)$ be the Borel $\sigma$-algebra on $U$.
  Since for every $(t,\omega)$ the function
  $x\mapsto \sigma_t(x)(\omega)$ is continuous on $U$, the function
  $(t,\omega,x)\mapsto \sigma_t(x)(\omega)$ is
  $\mathcal{P}\times \mathcal{B}$-measurable. It follows that
  \begin{align*}
    E &\set \descr{(t,\omega,x)}{\rank
        {\kappa^{X}_t(\omega)}\sigma_t(x)(\omega) < \rank
        \kappa^{X}_t(\omega)} 
        \in \mathcal{P}\times \mathcal{B}.
  \end{align*}
  From Fubini's theorem we deduce the equivalences:
  \begin{displaymath}
    \mM{(\mu^{X}\times \lambda)}{E} = 0 \quad \Leftrightarrow \quad
    \mM{\mu^{X}}{F}=0 \quad \Leftrightarrow \quad \mM{\lambda}{I} = 0,  
  \end{displaymath}
  where
  \begin{displaymath}
    F \set \descr{(t,\omega)}{\mM{\lambda}{\descr{x\in U}{\rank
          {\kappa^{X}_t(\omega)}\sigma_t(x)(\omega) < 
          \rank \kappa^{X}_t(\omega)}}>0}. 
  \end{displaymath}
  Hence to obtain the multi-dimensional version of the theorem we need
  to show that $\mu^{X}(F)=0$.

  From Lemma~\ref{lem:5} and the analyticity of the function
  $x\mapsto \sigma_t(x)(\omega)$ we deduce that
  \begin{equation}
    \label{eq:6}
    F = \descr{(t,\omega)}{\rank
      {\kappa^{X}_t(\omega)}\sigma_t(x)(\omega) < 
      \rank \kappa^{X}_t(\omega), \; \forall x\in U}. 
  \end{equation}
  
  We recall now that if $(x_n)$ is a sequence in $U$ that converges to
  $x_0$, then the martingales $(R(x_n),Y(x_n))$ converge to the
  martingale $(R(x_0),Y(x_0)) = (S(x_0),1)$ in $\mathcal{L}_1$. By
  Lemma~\ref{lem:7}, passing to a subsequence, we can assume that
  $(R(x_n),Y(x_n))\to (R(x_0),Y(x_0))$ in $\mathcal{H}_{1,loc}$. From
  Lemma~\ref{lem:3} we deduce that
  \begin{align*}
    {\kappa^{X}}\alpha(x_n)
    &\overset{\mu^{X}}{\rightarrow} 0, \\
    {\kappa^{X}}\beta(x_n) &\overset{\mu^{X}}{\rightarrow}
                             {\kappa^{X}}\beta(x_0) = {\kappa^{X}}\sigma(x_0).
  \end{align*}
  It follows that
  \begin{displaymath}
    {\kappa^{X}}\sigma(x_n) = {\kappa^{X}}(\beta(x_n) -
    \alpha(x_n) S^*_{-}(x_n)) 
    \overset{\mu^{X}}{\rightarrow} {\kappa^{X}}\beta(x_0) =
    {\kappa^{X}}\sigma(x_0).  
  \end{displaymath}
  Passing to a subsequence we can choose the sequence $(x_n)$ so that
  \begin{displaymath}
    {\kappa^{X}}\sigma(x_n) \to {\kappa^{X}}\sigma(x_0),
    \quad \mu^{X}-a.s..    
  \end{displaymath}
  As $a\mapsto \rank a$ is a lower-semicontinuous function on
  matrices, it follows that
  \begin{displaymath}
    \liminf_n \rank {\kappa^{X}} \sigma(x_n) \geq \rank {\kappa^{X}}
    \sigma(x_0), \quad \mu^{X}-a.s..  
  \end{displaymath}
  Accounting for~\eqref{eq:6} we obtain that
  \begin{displaymath}
    F \subset D(x_0), \quad \mu^{X}-a.s.. 
  \end{displaymath}
  However, as $S(x_0)$ has the MRP, Lemma~\ref{lem:4} yields that
  $\mM{\mu^{X}}{D(x_0)}=0$ and the multi-dimensional version of the
  theorem follows.
    
  Assume now that $U$ is an open interval in $\mathbf{R}$ and that
  contrary to the assertion of the theorem the exception set $I$ is
  uncountable. Then there are $\epsilon >0$, a closed interval
  $[a,b] \subset U$, and a sequence $(x_n) \subset [a,b]$ such that
  \begin{displaymath}
    \mM{\mu^{X}}{D(x_n)}
    \geq \epsilon, \quad n\geq 1.  
  \end{displaymath}
  Since for every $(t,\omega)$ the function
  $x\mapsto \sigma_t(x)(\omega)$ is analytic, we deduce from
  Lemma~\ref{lem:5} that on every closed interval the integer-valued
  function $x\mapsto \rank(\kappa^X_t(\omega)\sigma_t(x)(\omega))$ has
  constant value except for a \emph{finite} number of points, where
  its values are smaller. Hence, if
  \begin{displaymath}
    \rank(\kappa^X_t(\omega)\sigma_t(x_{n'})(\omega)) <
    \rank(\kappa^X_t(\omega)) \text{ for countable } (n')\subset (n), 
  \end{displaymath}
  then
  \begin{displaymath}
    \rank(\kappa^X_t(\omega)\sigma_t(x)(\omega)) <
    \rank(\kappa^X_t(\omega)) \text{ for all } x\in U. 
  \end{displaymath}
  Accounting for~\eqref{eq:6} it follows that
  \begin{displaymath}
    \limsup_n D(x_n) \set \cap_n \cup_{m\geq n} D(x_m) = F 
  \end{displaymath}
  and thus
  \begin{displaymath}
    \mM{\mu^{X}}{F} \geq \limsup_n \mM{\mu^{X}}{D(x_n)} \geq \epsilon.
  \end{displaymath}  
  However, as we have already shown, $\mM{\mu^{X}}{F}=0$ and we arrive
  to a contradiction.
\end{proof}

\appendix

\section{Analytic fields of martingales and stochastic integrals}
\label{sec:analyt-fields-mart}

We denote by $\mathbf{D}^{\infty}([0,\infty),\mathbf{R}^d)$ the Banach
space of RCLL (right-continuous with left limits) functions
$\map{f}{[0,\infty)}{\mathbf{R}^d}$ equipped with the uniform norm:
$\norm{f}_{\infty} \set \sup_{t\geq 0} \abs{f(t)}$.

\begin{Theorem}
  \label{th:3}
  Let $U$ be an open connected set in $\mathbf{R}^l$ and
  $x\mapsto \xi(x)$ be an analytic map from ${U}$ to
  $\mathcal{L}_1(\mathbf{R}^d)$. Then there are modifications of the
  accompanying $d$-dimensional martingales
  \begin{displaymath}
    M_t(x) \set \cEP{t}{\xi(x)}, 
  \end{displaymath}
  such that for every $\omega\in \Omega$ the maps
  $x\mapsto M_{\cdot}(x)(\omega)$ taking values in
  $\mathbf{D}^{\infty}([0,\infty), \mathbf{R}^d)$ are analytic on $U$.

  If in addition, the MRP holds for a local martingale $X$ with values
  in $\mathbf{R}^m$, then there is a stochastic field
  $x\mapsto \sigma(x)$ of integrands for $X$ such that
  \begin{displaymath}
    M(x)  = M_0(x) + \sigma(x) \cdot X, 
  \end{displaymath}
  and for every $(t,\omega)\in [0,\infty)\times \Omega$ the function
  $x\mapsto \sigma_t(x)(\omega)$ taking values in $m\times d$-matrices
  is analytic on $U$.
\end{Theorem}

The proof of the theorem is divided into a series of lemmas. For a
multi-index $\alpha = (\alpha_1,\dots,\alpha_l) \in \mathbf{Z}^l_+$ we
denote
\begin{displaymath}
  \abs{\alpha} \set \alpha_1+\dots+\alpha_l. 
\end{displaymath}
The space $\mathcal{H}_1$ has been introduced just before
Lemma~\ref{lem:3}.

\begin{Lemma}
  \label{lem:6}
  Let $(M^\alpha)_{\alpha \in \mathbf{Z}_+^l}$ be uniformly integrable
  martingales with values in $\mathbf{R}^d$ such that
  \begin{displaymath}
    \sum_{\alpha} 2^{\abs{\alpha}}\sNorm{M^\alpha}{\mathcal{L}_1} < \infty. 
  \end{displaymath}
  Then there is an increasing sequence $(\tau_m)$ of stopping times
  such that $\braces{\tau_m=\infty}\uparrow \Omega$ and
  \begin{displaymath}
    \sum_{\alpha} \sNorm{M^{\alpha,\tau_m}}{\mathcal{H}_1} <
    \infty, \quad m\geq 1. 
  \end{displaymath}
\end{Lemma}

\begin{proof}
  We define the martingale
  \begin{align*}
    L_t \set \cEP{t}{\sum_\alpha 2^{\abs{\alpha}}
    \abs{M^\alpha_\infty}}, \; t\geq 0,
  \end{align*}
  and stopping times
  \begin{displaymath}
    \tau_m \set \inf\descr{t\geq 0}{L_t \geq m}, \; m\geq 1. 
  \end{displaymath}
  Clearly, $\braces{\tau_m=\infty} \uparrow \Omega$ as $m\to \infty$
  and $\abs{M^\alpha} \leq 2^{-\abs{\alpha}} L$. Moreover,
  \begin{displaymath}
    \sNorm{L^{\tau_m}}{\mathcal{H}_1} = \EP{\sup_{0\leq t\leq \tau_m}
      L_t} \leq m + \EP{L_{\tau_m}} = m + L_0 < \infty. 
  \end{displaymath}
  It follows that
  \begin{displaymath}
    \sum_\alpha \sNorm{M^{\alpha,\tau_m}}{\mathcal{H}_1}
    \leq \sNorm{L^{\tau_m}}{\mathcal{H}_1}
    \sum_\alpha 2^{-\abs{\alpha}} < \infty.
  \end{displaymath}
\end{proof}

\begin{Lemma}\label{lem:7}
  Let $(M^n)$ and $M$ be uniformly integrable martingales such that
  $M^n \to M$ in $\mathcal{L}_1$. Then there exists a subsequence of
  $(M^n)$ that converges to $M$ in $\mathcal{H}_{1,loc}$.
\end{Lemma}
\begin{proof}
  Since $M^n \to M$ in $\mathcal{L}_1$ there exists a subsequence
  $(M^{n_k})$ such that
  \begin{displaymath}
    \sum_{k=1}^\infty \sNorm{M^{n_{k+1}}-M^{n_k}}{\mathcal{L}_1} 2^k  < \infty.
  \end{displaymath}
  Lemma~\ref{lem:6} implies that $M^{n_k}\to M$ in
  $\mathcal{H}_{1,loc}$.
\end{proof}

Let $X$ be a square integrable martingale taking values in
$\mathbf{R}^m$. As in Section~\ref{sec:analyticity-mr} we associate
with $X$ the increasing predictable process $A^X\set \trace \pvar{X}$,
the $\mathcal{S}^m_+$-valued predictable process $\kappa^{X}$ such
that $\pvar{X} = (\kappa^{X})^2 \cdot A^X$, and a finite measure
$\mu^{X}(dt,d\omega) \set dA^X_t(\omega) \ProbP{d\omega}$ on the
predictable $\sigma$-algebra $\mathcal{P}$ of
$[0,\infty) \times \Omega$.  We recall that an integrand $\gamma$ for
$X$ is \emph{minimal} if
\begin{equation}
  \label{eq:7}
  \gamma = {\kappa^X}^{\oplus} {\kappa^X} \gamma. 
\end{equation}

\begin{Lemma}
  \label{lem:8}
  Let $X$ be a bounded martingale with values in $\mathbf{R}^m$ and
  $(\gamma^\alpha)_{\alpha\in \mathbf{Z}^l_+}$ be minimal integrands
  for $X$ such that
  \begin{equation}
    \label{eq:8}
    \sum_\alpha \sNorm{\gamma^\alpha\cdot X}{\mathcal{H}^1} <
    \infty. 
  \end{equation}
  Then
  \begin{equation}
    \label{eq:9}
    \sum_\alpha \abs{\gamma^\alpha}^2 =  \sum_\alpha
    \abs{\gamma^\alpha_t(\omega)}^2 < \infty, \quad 
    \mu^{X}-a.s.. 
  \end{equation}

\end{Lemma}
\begin{proof}
  By Davis' inequality, \eqref{eq:8} is equivalent to
  \begin{displaymath}
    \sum_\alpha \EP{\left[\gamma^\alpha\cdot X\right]_\infty^{1/2}} <
    \infty. 
  \end{displaymath}
  By replacing if necessary $\gamma^\alpha$ with
  $\frac1{1+\abs{\gamma^\alpha}}\gamma^\alpha$, we can assume without
  a loss of generality that $\abs{\gamma^\alpha}\leq 1$. Let us show
  that in this case the increasing optional process
  \begin{displaymath}
    B_t \set \sum_\alpha \left[\gamma^\alpha\cdot
      X\right]_t, \quad t\geq 0, 
  \end{displaymath}
  is locally integrable. Since
  \begin{displaymath}
    B_\infty = \sum_\alpha \left[\gamma^\alpha\cdot
      X\right]_\infty \leq \Bigl(\sum_\alpha \left[\gamma^\alpha\cdot
      X\right]_\infty^{1/2}\Bigr)^2 < 
    \infty,  
  \end{displaymath}
  we only need to check that the positive jump process $\Delta B$ is
  locally integrable. Actually, we shall show that
  $\sup_{t\geq 0} \Delta B_t$ is integrable. Indeed, as $X$ is
  bounded, there is a constant $c>0$ such that
  $\abs{(\gamma^\alpha)^* \Delta X} \leq c$. Hence,
  \begin{displaymath}
    \sup_{t\geq 0}\Delta B_t \leq  \sum_\alpha ((\gamma^\alpha)^*\Delta X)^2 \leq
    c\sum_{\alpha} \abs{(\gamma^\alpha)^*{\Delta X}} \leq  c\sum_\alpha
    \left[\gamma^\alpha\cdot X\right]_\infty^{1/2}, 
  \end{displaymath}
  where the right-hand side has finite expected value.

  Since for every stopping time $\tau$
  \begin{displaymath}
    \EP{B_\tau} = \sum_\alpha \EP{\left[\gamma^\alpha\cdot X\right]_\tau}
    = \sum_\alpha \EP{\int_0^\tau
      \abs{{\kappa^X}\gamma^{\alpha}}^2 dA^X}, 
  \end{displaymath}
  the local integrability of $B$ yields the existence of stopping
  times $(\tau^m)$ such that $\tau_m \uparrow \infty$ and
  \begin{displaymath}
    \sum_\alpha \EP{\int_0^{\tau_m}
      \abs{{\kappa^X}\gamma^\alpha}^2  dA^X} = \sum_\alpha
    \mE{\mu^X}{\abs{{\kappa^X}\gamma^\alpha}^2 \setind{[0,\tau^m]}}
    <\infty.  
  \end{displaymath}
  It follows that
  \begin{displaymath}
    \sum_\alpha \abs{{\kappa^X}\gamma^\alpha}^2 < \infty, \quad
    \mu^X-a.s..  
  \end{displaymath}
  This convergence implies~\eqref{eq:9} in view of
  inequalities~\eqref{eq:4} for minimal integrands.
\end{proof}

\begin{Lemma}
  \label{lem:9}
  Let $X$ be a square integrable martingale taking values in
  $\mathbf{R}^m$ and $(\gamma^n)$ be minimal integrands for $X$ such
  that $(M^n \set \gamma^n\cdot X)$ are uniformly integrable
  martingales. Suppose that there are a uniformly integrable
  martingale $M$ and a predictable process $\gamma$ such that
  $M^n \to M$ in $\mathcal{L}_1$ and
  $\gamma^n_t(\omega) \to \gamma_t(\omega)$ for every $(t,\omega)$.
  Then $\gamma$ is a minimal integrand for $X$ and $M=\gamma\cdot X$.
\end{Lemma}
\begin{proof}
  In view of characterization~\eqref{eq:7} for minimal integrands, the
  minimality of every element of $(\gamma^n)$ implies the minimality
  of $\gamma$ provided that the latter is $X$-integrable. Thus we only
  need to show that $\gamma$ is $X$-integrable and $M=\gamma\cdot X$.

  By Lemma~\ref{lem:7}, passing to subsequences, we can assume that
  $M^n = \gamma^n\cdot X\to M$ in $\mathcal{H}_{1,loc}$. Since the
  space of stochastic integrals is closed under the convergence in
  $\mathcal{H}_{1,loc}$, there is a $X$-integrable predictable process
  $\widetilde{\gamma}$ such that $M = \widetilde \gamma \cdot X$.
  From Lemma~\ref{lem:3} we deduce that
  \begin{displaymath}
    \kappa^X (\gamma^n - \widetilde \gamma)
    \overset{\mu^{X}}{\rightarrow} 0. 
  \end{displaymath}
  It follows that
  \begin{displaymath}
    \kappa^X(\widetilde \gamma - \gamma) = 0, \quad \mu^X-a.s.,
  \end{displaymath}
  and Lemma~\ref{lem:2} yields the result.
\end{proof}

\begin{proof}[Proof of Theorem~\ref{th:3}.]
  It is sufficient to prove the existence of the required analytic
  versions only locally, in a neighborhood of every $y\in U$.
  Hereafter, we fix $y\in U$.  There are
  $\epsilon = \epsilon(y) \in (0,1)$ and a family
  $(\zeta_\alpha =\zeta_\alpha(y))_{\alpha \in \mathbf{Z}^l_+}$ in
  $\mathcal{L}_1$ such that
  \begin{gather*}
    \xi(x) = \xi(y) + \sum_\alpha \zeta_\alpha (x-y)^\alpha, \quad
    \max_i\abs{x_i-y_i}<2\epsilon, \\
    \sum_{\alpha} \EP{\abs{\zeta_\alpha}}(2\epsilon)^{\abs{\alpha}} <
    \infty,
  \end{gather*}
  where the first series converges in $\mathcal{L}_1$.

  By taking conditional expectations with respect to $\mathcal{F}_t$
  we obtain that
  \begin{equation}
    \label{eq:10}
    M_t(x) = M_t(y) + \sum_{\alpha} L^{\alpha}_{t} (x-y)^\alpha, \quad
    \max_i\abs{x_i-y_i} < 2\epsilon, 
  \end{equation}
  where $L^{\alpha}_{t} \set \cEP{t}{\zeta_\alpha}$ and the series
  converges in $\mathcal{L}_1$. Lemma~\ref{lem:6} yields an increasing
  sequence $(\tau_m)$ of stopping times such that
  $\braces{\tau_m=\infty}\uparrow \Omega$ and
  \begin{displaymath}
    \sum_{\alpha}\sNorm{L^{\alpha,\tau_m}}{\mathcal{H}_1} \epsilon^{\abs{\alpha}}<
    \infty, \quad m\geq 1. 
  \end{displaymath} 
  It follows that
  \begin{displaymath}
    \sum_\alpha \sup_{t\geq 0} \abs{L^{\alpha}_{t}(\omega)} \epsilon^{\abs{\alpha}} <
    \infty, \quad \mathbb{P}-a.s.
  \end{displaymath}
  and we can modify the martingales $(L^\alpha)$ so that the above
  convergence holds true for every $\omega\in \Omega$. Then the series
  in~\eqref{eq:10} converges {uniformly} in $t$ for every
  $\omega\in \Omega$ and every $x$ such that
  $\max_i \abs{x_i - y_i}<\epsilon$. Thus, it defines the
  modifications of $M(x)$ for such $x$ with the required analytic
  properties.

  For the second part of the theorem we observe that the statement is
  invariant with respect to the choice of the local martingale $X$
  that has the MRP. Proposition~2 in~\cite{KramSirb:06} shows that we
  can choose $X$ to be a bounded $m$-dimensional martingale.

  As $X$ has the MRP, there are minimal integrands $\sigma(y)$ and
  $(\gamma^\alpha)$ such that
  \begin{align*}
    M(y) & = M_0(y) + \sigma(y)\cdot X, \\
    L^\alpha & = L^{\alpha}_{0} + \gamma^\alpha \cdot X, \quad \alpha
               \in \mathbf{Z}^l_+.
  \end{align*}
  From~Lemma~\ref{lem:8} we deduce that
  \begin{displaymath}
    \sum_\alpha \abs{\gamma^{\alpha}_t(\omega)}^2
    \epsilon^{2\abs{\alpha}}   < \infty  
  \end{displaymath}
  for all $(t,\omega)$ except a predictable set of $\mu^X$-measure
  $0$. By Lemma~\ref{lem:2} we can set $\gamma^\alpha=0$ on this set
  without changing $\gamma^\alpha\cdot X$. Then the series converges
  for every $(t,\omega)$. As $\epsilon \in (0,1)$, we deduce that
  \begin{displaymath}
    \sum_\alpha 
    \abs{\gamma^{\alpha}_t(\omega)} \epsilon^{2\abs{\alpha}} < \infty 
  \end{displaymath}
  and thus for $x=(x_1,\dots,x_l)$ such that
  $\max_i \abs{x_i - y_i}<\epsilon^2$ and every $(t,\omega)$ we can
  define
  \begin{displaymath}
    \sigma_t(x)(\omega) \set \sigma_t(y)(\omega) + \sum_\alpha
    \gamma^\alpha_t(\omega) (x-y)^\alpha. 
  \end{displaymath}
  By construction, the function $x\to \sigma_t(x)(\omega)$ is analytic
  in a neighborhood of $y$.  By Lemma~\ref{lem:9}, for every $x$ such
  that $\max_i \abs{x_i-y_i}<\epsilon^2$ the predictable process
  $\sigma(x)$ is an integrand for $X$ and
  \begin{align*}
    M(x) &= M(y) + \sum_\alpha L^\alpha (x-y)^\alpha \\
         &= M_0(x) + \sigma(y)\cdot X + \sum_\alpha (\gamma^\alpha\cdot X)
           (x-y)^\alpha \\
         &= M_0(x) + \sigma(x)\cdot X.
  \end{align*}
\end{proof}

\section{The MRP under the change of measure}
\label{sec:mrp-under-change}

Let $X$ be a $d$-dimensional local martingale and $Z>0$ be the density
process of $\widetilde{\mathbb{P}} \sim \mathbb{P}$. We denote by
$\widetilde Z \set 1/{Z}$ the density process of $\mathbb{P}$ under
$\widetilde{\mathbb{P}}$ and set $L \set \widetilde{Z}_{-}\cdot Z$ and
$\widetilde{L} \set {Z}_{-}\cdot \widetilde{Z}$. Using integration by
parts we deduce that
\begin{displaymath}
  d(\widetilde Z X) = X_{-}d\widetilde{Z} + \widetilde{Z}_{-}
  d\widetilde{X}, 
\end{displaymath}
where
\begin{displaymath}
  \widetilde X = X +
  \qcov{X}{\widetilde L}. 
\end{displaymath}
It follows that $\widetilde X$ is a $d$-dimensional local martingale
under $\widetilde{\mathbb{P}}$. Of course, this is just a version of
Girsanov's theorem.

We observe that the relations between $X$ and $\widetilde X$ are
symmetric in the sense that
\begin{displaymath}
  X = \widetilde X + \qcov{\widetilde X}{L}.
\end{displaymath}
Indeed, as we have already shown,
$Y \set \widetilde X + \qcov{\widetilde X}{L}$ is a $d$-dimensional
local martingale. Clearly, the local martingales $X$ and $Y$ have the
same initial values and the same continuous martingale parts. Finally,
they have identical jumps:
\begin{align*}
  \Delta(Y-X) &= \Delta (\qcov{X}{\widetilde L} + \qcov{\widetilde
                X}{L}) = \Delta X(\Delta\widetilde{L} + \Delta L + \Delta
                \widetilde{L} \Delta L) \\
              &= \Delta X \Delta(Z\widetilde Z) = 0.
\end{align*}

\begin{Theorem}
  \label{th:4}
  The local martingale $X$ has the MRP if and only if the local
  martingale $\widetilde X$ under $\widetilde{\mathbb{P}}$ has the
  MRP.
\end{Theorem}

\begin{proof}
  By symmetry, it is sufficient to prove only one of the
  implications. We assume that $X$ has the MRP. Let $\widetilde M$ be
  a local martingale under $\widetilde{\mathbb{P}}$. The arguments
  before the statement of the theorem yield the unique local
  martingale $M$ such that
  \begin{displaymath}
    \widetilde M =  M + \qcov{M}{\widetilde L}. 
  \end{displaymath}
  If now $H$ is an integrand for $X$ such that $M = M_0 + H\cdot X$,
  then
  \begin{displaymath}
    \widetilde M = \widetilde M_0 + H\cdot (X + \qcov{X}{\widetilde
      L}) = \widetilde M_0 + H\cdot \widetilde X. 
  \end{displaymath}
\end{proof}
 
\bibliographystyle{plainnat}

\bibliography{finance.bib}

\end{document}